\documentclass[11pt, reqno]{amsart}

\usepackage{amssymb}
\usepackage{graphicx}
\usepackage{xcolor} % A package to add color.
\usepackage{tensor}

\usepackage{fullpage} % Sets all margins to 1 in.  
\usepackage{amsmath}
\usepackage{amsthm}
\usepackage{verbatim}
\usepackage{hyperref}
\usepackage{todonotes, enumitem} 
\usepackage{array} 
\usepackage{enumitem}
\setlist[enumerate]{leftmargin=1.8em}
\setlist[itemize]{leftmargin=2.5em}
\usepackage{seqsplit,mathtools}

\definecolor{green}{rgb}{0,0.8,0} % Redefines the color green.
\newtheorem{theorem}{Theorem}[section]
\newtheorem{corollary}[theorem]{Corollary}
\newtheorem{lemma}[theorem]{Lemma}

\newtheorem{proposition}[theorem]{Proposition}

\theoremstyle{definition}

\theoremstyle{remark}

\numberwithin{equation}{section}
\newcommand{\nnrm}[1]{{\vert\kern-0.25ex\vert\kern-0.25ex\vert #1 
		\vert\kern-0.25ex\vert\kern-0.25ex\vert}}

\newcommand{\lap}{\Delta}

\newcommand{\rd}{\partial}
\newcommand{\nb}{\nabla}

\newcommand{\omg}{\omega}

\newcommand{\ackn}[1]{
	\addtocontents{toc}{\protect\setcounter{tocdepth}{1}}
	\subsection*{Acknowledgements} {#1}
	\addtocontents{toc}{\protect\setcounter{tocdepth}{1}} }

\definecolor{purple}{rgb}{0.65, 0, 1}
\definecolor{orange}{rgb}{1,.5,0}

\allowdisplaybreaks

\begin{document}
		
	\title{Growth estimates for axisymmetric Euler\\ equations without swirl}
	
	\author{Khakim Egamberganov}
	\address{Department of Mathematics, National University of Singapore, Block S17, 10 Lower Kent Ridge Road, Singapore, 119076, Singapore.}
	\email{k.egamberganov@u.nus.edu}
		
	\author{Yao Yao}
	\address{Department of Mathematics, National University of Singapore, Block S17, 10 Lower Kent Ridge Road, Singapore, 119076, Singapore.}
	\email{yaoyao@nus.edu.sg}
	
	\date{}
	
	\maketitle
	
	\renewcommand{\thefootnote}{\fnsymbol{footnote}}
	%\footnotetext{\emph{2020 AMS Mathematics Subject Classification:} 35Q35}
	%\footnotetext{\emph{Key words: vorticity confinement, } }
	\renewcommand{\thefootnote}{\arabic{footnote}}

\begin{abstract} We consider the axisymmetric Euler equations in $\mathbb{R}^3$ without swirl, and establish several upper and lower bounds for the growth of solutions. On the one hand, we obtain an upper bound $t^2$ for the radial moment $\int_{\mathbb{R}^3} r\omega^\theta dx$, which is the conjectured optimal rate by Childress (Phys. D 237(14-17):1921-1925, 2008). On the other hand, for all initial data satisfying certain symmetry and sign conditions, we prove that the radial moment grows at least like $t/\log t$ as time goes to infinity, and $\|\omega(\cdot,t)\|_{L^p(\mathbb{R}^3)}$ exhibits at least $t^{1/4}$ growth in the limsup sense for all $1\leq p\leq \infty$. To the best of our knowledge, this is the first result to establish power-law $L^p$-norm growth for smooth, compactly supported initial vorticity in $\mathbb{R}^3$.
For these initial data, we also show that nearly all vorticity must eventually escape to $r\to\infty$ in the time-integral sense.
\end{abstract}

\section{Introduction}
The incompressible Euler equations in $\mathbb{R}^3$ describe the motion of incompressible and inviscid fluids in three dimensions. The equations can be written in vorticity form as
\begin{equation} \label{eq:Euler}
	\left\{
	\begin{aligned}
	&\rd_t\omg + u\cdot\nb\omg = \omega \cdot \nabla u,  \\
	&u = \nb \times (-\lap)^{-1}\omg, 
	\end{aligned}
	\right.
\end{equation}
where $u(\cdot,t)$  and $\omega(\cdot,t): = \nabla \times u(\cdot,t)$ denote the velocity and vorticity of the fluid, respectively. At each time, the velocity $u$ can be recovered from the vorticity $\omega$ through the second equation, which is known as the Biot--Savart law. The term $\omega \cdot \nabla u$ on the right hand side of the first equation is referred to as the vortex stretching term, which can cause vorticity growth and might potentially lead to singularities.

In this paper, we restrict ourselves to axisymmetric Euler flows in $\mathbb{R}^3$ without swirl. These are the solutions of \eqref{eq:Euler} where the velocity field $u$ takes the form
\[
u(r,z,t) = u_r (r,z,t) e_r+ u_z(r,z,t) e_z
\]
when expressed in cylindrical coordinates $(r,\theta,z)$, where $(e_r,e_{\theta},e_z)$ are the cylindrical unit vectors in $\mathbb{R}^3$. 
 Then the vorticity becomes 
\begin{equation}\label{omega}
\omega (r,z,t) = \omega^{\theta} (r,z,t) e_{\theta} \quad \text{with} \quad \omega^{\theta} := -\partial_z u_{r} + \partial_r u_z,
\end{equation}
and the first equation in \eqref{eq:Euler} simplifies to 
\begin{equation} \label{eq:Euler2}
\left(\partial_t  + u^r\partial_r + u^z\partial_z\right) \frac{\omega^{\theta}}{r} = 0.
\end{equation}
This tells us that the quantity $\omega^{\theta}/r$ remains conserved along particle trajectories, and consequently, $\|\omega^{\theta}/r\|_{L^p(\mathbb{R}^3)}$ is conserved in time for all $1\leq p\leq \infty$.

		\medskip
Before stating our results, we begin with a brief review of the existing literature on well-posedness and growth results for the axisymmetric Euler equations in $\mathbb{R}^3$ without swirl.

\medskip
\noindent\textbf{Global well-posedness and singularity for 3D axisymmetric Euler  without swirl.}
For the 3D axisymmetric Euler equation without swirl, global well-posedness of the solution was first shown by Ukhovsksii--Yudovich \cite{UY} in 1968, for $u_0\in H^3(\mathbb{R}^3)$ and $\omega_0^\theta/r \in (L^2\cap L^\infty)(\mathbb{R}^3)$. Subsequently, the regularity requirements on $u_0$ were relaxed by Majda \cite{majda1986vorticity}, Saint Raymond \cite{saint1994remarks} and Shirota--Yanagisawa \cite{SY} to $u_0 \in H^s$ with $s>5/2$. 

\smallskip
 The first well-posedness result using only $L^p$ bounds on $\omega_0^\theta$ (without any additional regularity of $u_0$) is obtained by Danchin \cite{Danchin}, under the assumption $\omega_0^\theta \in (L^{3,1}\cap L^\infty)(\mathbb{R}^3)$ and $\omega_0^\theta/r \in L^{3,1}(\mathbb{R}^3)$, where $L^{3,1}$ denotes the Lorentz space. Here the condition $\omega_0^\theta/r \in L^{3,1}(\mathbb{R}^3)$ is necessary -- Kim--Jeong \cite{JK} showed that there exists some $q_0<\infty$ such that the equation is ill-posed if the condition is replaced by $\omega_0^\theta/r \in L^{3,q}(\mathbb{R}^3)$ for any $q\in (q_0,\infty]$. The well-posedness theory has been further extended by Abidi--Hmidi--Keraani \cite{AHK} when $u_0$ belongs to the critical Besov spaces $B_{p,1}^{1+3/p}$ and $\omega_0^\theta/r\in L^{3,1}(\mathbb{R}^3)$. 

\smallskip
On the other hand, if $\omega_0^\theta/r$ is more singular near the axis, finite-time singularity can happen: Elgindi \cite{Elgindi} and Elgindi--Ghoul--Masmoudi \cite{EGM} proved finite-time
singularity formation of \eqref{eq:Euler}--\eqref{omega} for some $\omega_0^\theta \in C^{\alpha}(\mathbb{R}^3)$ with sufficiently small $\alpha>0$, and a new proof was recently found by C\'{o}rdoba--Mart\'{i}nez-Zoroa--Zheng \cite{CMZ} for some initial data $u_0 \in C^\infty(\mathbb{R}^3\setminus\{0\})\cap C^{1, \alpha}\cap L^2$ for small $\alpha>0$.

\medskip
\noindent\textbf{Vorticity growth from anti-parallel vortex rings.}
For initial data belonging to the class that leads to global well-posedness, it is natural to ask whether the vorticity can grow as $t\to\infty$. Since $\omega^\theta/r$ is conserved along each particle trajectory by \eqref{eq:Euler2}, the vorticity would only grow to infinity if the radius of some fluid particles grow to infinity. When the initial vorticity consists of two anti-parallel vortex rings of the form
\begin{equation}\label{temp_sym}
\omega_0^\theta(r,z) = f(r,z) - f(r,-z)
\end{equation}
where $f\geq 0$ is localized near a point $(r_0,z_0)$ with $r_0>0, z_0>0$, experimental evidence shows that the two vortex rings move toward each other and expand over time as their vorticity amplitude increases; see \cite[Section 5.2]{shariff1992vortex} and the references therein. 

\smallskip
This scenario has been formally analyzed in a series of work by Childress, Gilbert and Valiant \cite{Childress1, CGV, CG}. Childress \cite{Childress1} conjectured that this scenario leads to the fastest possible growth of vorticity amplitude: for $t\gg 1$, formal analysis suggests that there is a region in the $rz$-plane of length scale $t^{-1}$ and located at $r\sim t^{4/3}$, where the vorticity amplitude is of order $t^{4/3}$,  and the kinetic energy in this region remains of order 1. See Figure~\ref{fig_dipole} for an illustration of the conjectured scenario, with the region contained within the dashed orange rectangle. Numerical simulations by Childress--Gilbert--Valiant \cite{CGV} suggest that the $t^{4/3}$ vorticity growth can be achieved, although rigorously justifying these asymptotics remains a challenging open question.

\begin{figure}[htbp]
\includegraphics[scale=1.1]{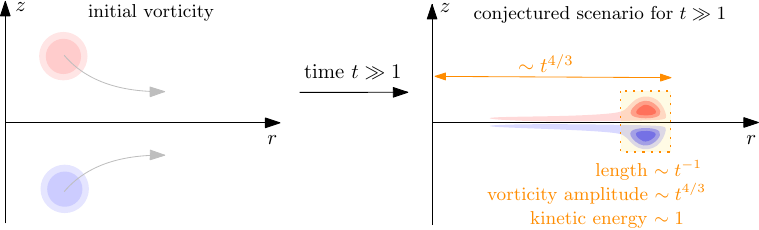}
\caption{\label{fig_dipole} The conjectured vorticity growth scenario in Childress \cite{Childress1}, where $\|\omega^\theta\|_{L^\infty}$ is conjectured to have $t^{4/3}$ growth. See Childress--Gilbert--Valiant \cite[Figure~7]{CGV} for numerical simulations of this scenario.}
\end{figure}
\medskip
\noindent\textbf{Upper bounds of vorticity growth.}
For compactly supported initial vorticity satisfying $\omega_0^\theta/r \in L^\infty(\mathbb{R}^3)$, Childress \cite{Childress2} established the upper bound $\|\omega^\theta\|_{L^\infty} \lesssim t^2$ by solving a variational optimization problem that maximizes the radial velocity subject to constraints on the vorticity. Alternatively, this  $t^2$ upper bound also be derived using the interpolation inequality for maximum velocity in Feng--\v{S}ver\'{a}k \cite[Eq. (1.23)]{FeSv}; see \cite[Section 1.4.2]{LJ} for an explanation. 

\smallskip
For compactly supported vorticity with $\omega_0^\theta/r \in L^\infty(\mathbb{R}^3)$, Lim--Jeong \cite{LJ} has obtained the upper bound $\|\omega^\theta\|_{L^\infty} \lesssim t^{4/3}$, which matches the conjectured growth rate by Childress~\cite{Childress1}. They obtained this bound by deriving a novel interpolation inequality that bounds the radial velocity $u_r$ by the kinetic energy and conserved quantities involving the vorticity. Very recently, Shao--Wei--Zhang~\cite{SWZ} improved upon this by proving the same $t^{4/3}$ upper bound without the compact support assumption for $\omega_0^\theta$, for all initial data $u_0 \in L^2 \cap C^{1,\alpha}(\mathbb{R}^3)$ and $\omega_0^\theta,\, \omega_0^\theta/r \in L^\infty(\mathbb{R}^3)$.

\medskip
\noindent\textbf{Lower bounds of vorticity growth.}
When the initial vorticity satisfies the odd-in-$z$ assumption~\eqref{temp_sym} with $f \geq 0$ in $\{z > 0\}$ (which is the scenario in Figure~\ref{fig_dipole}, except that $f$ does not need to be localized), Choi--Jeong~\cite{Choi-Jeong} showed that the \emph{radial moment} $P(t) = \int_{\mathbb{R}^3} r |\omega^\theta|\, dx$ increases in time, while the \emph{vertical moment} $Z(t) = \int_{\mathbb{R}^3} z \omega^\theta \,dx$ decreases in time. By controlling $P'(t)$ from below using $Z(t)$ and the kinetic energy $E$, they also established a quantitative lower bound  $P(t) \gtrsim t^{\frac{2}{15}-}$. For patch-type initial vorticity, this directly leads to the estimate $\|\omega^\theta\|_{L^\infty} \gtrsim t^{\frac{1}{15}-}$. For certain initial vorticity in $C^{1,\gamma}$  (with $\gamma<\frac{1}{14}$) where $r/\omega_0^\theta$ satisfies some integrability condition in its support, the growth of $P(t)$ also implies a (slower) growth for $\|\omega^\theta\|_{L^\infty}$. 

More recently, Gustafson--Miller--Tsai \cite{GMT} has improved the lower bound into $P(t) \gtrsim t^{\frac{3}{4}-}$, based on new estimates for the time derivative of $P$, leading to the lower bound $\|\omega^\theta\|_{L^\infty} \gtrsim t^{\frac{3}{8}-}$ for patch-type initial data.

\smallskip
To the best of our knowledge, for \emph{smooth}, compactly supported initial vorticity $\omega_0$,  there was no example in the literature showing $\|\omega\|_{L^\infty}$ (or $\|\omega\|_{L^p}$) grows to infinity for the 3D incompressible Euler equation in the whole space $\mathbb{R}^3$, even without restricting to the axisymmetric setting. (For periodic domains, it is known that the ``$2+\frac{1}{2}$-dimensional flows''  can lead to exponential growth of $\|\omega\|_{L^\infty}$ in $\mathbb{T}^3$ \cite[Proposition 10.2]{elgindi2020ill}, and linear growth in $\mathbb{R}^2\times \mathbb{T}$ \cite[Section 2.3.1]{majda2002vorticity}.) When the domain has a boundary, or when $\omega_0$ is less regular, significant progress has been made  in the last decade on finite-time singularity formation \cite{EJ, Elgindi, EGM, CH1, CH2,  WLGB, CH3, CMZ}. For other infinite-in-time growth results on 3D Euler equation, see \cite{yudovich2000loss, Do, Choi-Jeong, choi2023filamentation, KPY}; these results are either in the presence of boundary, or the growth is on derivatives of $\omega$ instead of $\omega$ itself. We refer the readers to the reviews by Kiselev \cite{kiselev2018small} and Drivas--Elgindi \cite{drivas2023singularity} for an overview on singularity formation and growth in Euler equations.

\subsection{Our results}
  Let us start by setting the initial data we will be working on. Since the evolution of the vorticity vector $\omega$ is completely determined by the scalar component $\omega^{\theta}$, from now on, we will denote $\omega^\theta$ as $\omega$ (and the initial data $\omega^\theta_0$ as $\omega_0$), and refer to it as the vorticity.
Throughout this paper, we assume that
\begin{equation}
    \label{assu_w0}
   u_0 \in L^2(\mathbb{R}^3) \quad\text{ and } \omega_0, \frac{\omega_0}{r} \in L^1(\mathbb{R}^3)\cap L^\infty(\mathbb{R}^3).
\end{equation}
Note that there is a global-in-time solution in this class by  \cite{Danchin}, and if in addition $\omega_0$ is  assumed to be smooth, then $\omega(\cdot,t)$ is also smooth for all times.

In the axisymmetric setting, since the functions $u_r$, $u_z$, and $\omega$ can be written as functions of $(r,z)$ with $r\ge0$, it is sufficient to consider their evolution in the half plane
\[
\Pi:= \{ (r,z) \, : \, r\ge 0, \, \, z\in \mathbb{R} \}.
\]
For $k \in \mathbb{N}$, we define the \emph{$k$-th radial moment} as
\begin{equation} \label{eq:radial}
P_k(t) := \iint_{\Pi} r^k |\omega| \, drdz \quad  \quad \text{($k$-th radial moment)},
\end{equation}
and  the \textit{vertical moment} as
\begin{equation} \label{eq:vertical}
Z(t) := \iint_{\Pi} z\omega \, drdz  \quad  \quad \text{(vertical moment)}.
\end{equation}
These quantities play a key role in our study.
 
\medskip

Now, we state our main results, which can be divided into upper bounds and lower bounds. 

\medskip
\noindent\textbf{Upper bounds for the solutions}. We first prove an upper bound for all $k$-th radial moment for $k\geq 2$, and $\|\omega(\cdot,t)\|_{L^\infty}$. For this result, we do not require any additional symmetry or sign assumption on $\omega_0$.

\begin{theorem}[Upper bounds for radial moments, $\|u_r\|_{L^\infty}$ and $\|\omega\|_{L^\infty}$] \label{thm:main1}
Let the initial data $u_0$ and $\omega_0$ satisfy \eqref{assu_w0}. In addition, assume that $P_2(0)<\infty$. Then for each $k\geq 2$, if $P_k(0)<\infty$, we have
\begin{equation}\label{upper_bd_pk}
P_k(t) \le C(k,\omega_0, u_0) (1+t)^{\frac{4}{3}k-\frac{2}{3}} \quad\text{ for all }t\geq 0.
\end{equation}
In addition, we also have the following upper bounds for the velocity and vorticity growth:
\begin{equation}\label{upper_bd_u}
\|u_r(\cdot,t)\|_{L^\infty(\mathbb{R}^3)} \leq C(\omega_0,u_0) (1+t)^{1/3}  \quad\text{ for all }t\geq 0.
\end{equation}
\begin{equation}\label{upper_bd_omega}
\|\omega(\cdot,t)\|_{L^\infty(\mathbb{R}^3)} \leq C(\omega_0,u_0) (1+t)^{4/3}  \quad\text{ for all }t\geq 0.
\end{equation}
\end{theorem}

We point out that these upper bounds matches the conjectured optimal rate by Childress \cite{Childress1}, in the sense that if the conjectured vorticity growth scenario in Figure~\ref{fig_dipole} is true for some initial data, then the solution would have exactly these growth rates for $P_k$, $\|u_r\|_{L^\infty}$ and $\|\omega\|_{L^\infty}$. 

\smallskip
 Note that the bound $\|\omega\|_{L^\infty}\lesssim t^{4/3}$ in \eqref{upper_bd_omega} is the same as in Lim--Jeong \cite{LJ}, but we do not require $\omega_0$ to be compactly supported, and replace it by the more relaxed assumption $P_2(0)<\infty$. The same bound is also recently derived by Shao--Wei--Zhang \cite{SWZ} under the assumptions $u_0 \in L^2 \cap C^{1,\alpha}$ and $\omega_0, \omega_0/r \in L^\infty$. Compared to their result, we require less regularity, but more integrability assumptions.

\medskip
\noindent\textbf{Lower bounds for the solutions}.
Next we move on to deriving lower bounds for solutions, where we need to impose additional symmetry and sign conditions. In addition to \eqref{assu_w0}, we also require the initial vorticity to be odd in $z$, and and non-negative on the upper half-space $\{z>0\}$:
\begin{equation} \label{w0_odd}
\omega_0 (r,-z) = -\omega_0(r,z), \quad\text{ and }\omega_0 (r,z) \ge 0 \quad \text{for} \, \, z\ge 0.
\end{equation}
It is known that these properties are preserved for all $t$, thus it suffices to track the evolution of $\omega$ in the first quadrant \[
\Pi^{+}:= \{(r,z): \, r>0, z>0 \}.
\]

Under these additional symmetry and sign conditions, in the next theorem we show that $P_2(t)\gtrsim 1+\frac{t}{\log(2+t)}$, improving the previous bounds $P_2(t)\gtrsim t^{\frac{3}{4}-}$ in the literature \cite{Choi-Jeong, GMT}.

\begin{theorem}[$P_2(t)$ has at least $t^{1-}$-growth] \label{thm:main2}
In addition to \eqref{assu_w0}, assume that the initial data $\omega_0$ satisfies \eqref{w0_odd},  and $P_2(0)<\infty$, $Z(0)<\infty$. Then there exists a constant $c>0$ depending on $\omega_0$ and $u_0$, such that
\begin{equation} \label{sublinear_growth}
P_2(t) \ge P_2(0) + \frac{c(\omega_0,u_0) \, t}{\log (2+t)} \quad\text{ for all }t\geq 0.
\end{equation}
\end{theorem}

This immediately leads to the following corollary for patch-type initial data $\omega_0/r=1_D$:
\begin{corollary}[$t^{\frac{1}{2}-}$ growth of $\|\omega\|_{L^\infty}$ for patch-type solution]
Consider patch-type initial data $\omega_0/r=1_D$, where $D\subset \Pi^+$ is a compact set. Then the global-in-time solution $\omega(\cdot,t)$ (which is also of patch type) satisfies 
\[\|\omega(t)\|_{L^\infty} \geq c(\omega_0,u_0) \Big(1+ \frac{t}{\log(2+t)}\Big)^{1/2}.
\]
\end{corollary}
To see this, note that by conservation of $\int_{\Pi^+} \omega dr$, the maximum radius of vorticity support must grow like $P_2(t)^{1/2}$ at least. Since $\omega/r=1$ in the support for all times, it implies $\|\omega(t)\|_{L^\infty} \gtrsim P_2(t)^{1/2}$, which allows us to apply the lower bound of $P_2(t)$ in Theorem~\ref{thm:main2} to conclude.

\medskip
However, for smooth initial data, growth of $P_2(t)$ does not imply growth of $\|\omega\|_{L^\infty}$: if there is a large set where $\omega_0>0$  but has very low value, then one can have growth of $P_2$ if the vorticity in this set moves outwards, without causing growth in $\|\omega\|_{L^\infty}$. This is why Choi--Jeong \cite[Corollary~1.4]{Choi-Jeong} need to impose the condition that $r/\omega_0$ satisfies some integrability condition in its support, which only holds for certain $C^{1,\gamma}$ initial data for small $\gamma$.
Also, note that even for patch-type initial data, although the growth of $P_2$ implies growth for $\|\omega\|_{L^p(\mathbb{R}^3)}$ for sufficiently large $p$ \cite[Corollary 1.4]{Choi-Jeong},  the range cannot be extended for all $p\in[1,\infty]$. 

\medskip
In our next result, we aim to obtain the growth of $\|\omega\|_{L^p}$ using a completely different approach from \cite{Choi-Jeong, GMT}, without relying on the radial moment $P_2(t)$ at all. Instead, we use the evolution of the vertical moment $Z(t)$ in a crucial way. This lead to the following growth for all $L^p(\mathbb{R}^3)$ norms for $1\leq p\leq \infty$, for all initial data $\omega_0$  (including those smooth ones) satisfying \eqref{assu_w0}, \eqref{w0_odd}, and $Z(0)<\infty$.

\begin{theorem}[Growth of $\|\omega\|_{L^p(\mathbb{R}^3)}$ for all $1\leq p\leq \infty$] \label{thm:main3}
Under the assumptions \eqref{assu_w0}, \eqref{w0_odd}, and $Z(0)<\infty$, let $\omega(\cdot,t)$ be the corresponding global-in-time solution of \eqref{eq:Euler2}. Then there exists some constant $c(\omega_0,u_0)>0$ that only depends on the initial data, such that
\begin{equation} \label{limsup_growth}
\limsup_{t \to \infty} \frac{\|\omega (\cdot,t)\|_{L^{p}(\mathbb{R}^3)}}{t^{1/4}} \geq c(\omega_0,u_0)>0 \quad\text{ for all }1\leq p\leq \infty.
\end{equation}
\end{theorem}

Although the growth rate $t^{1/4}$ is not optimal, to the best of our knowledge, this seems to be the first example that gives unbounded growth of $\|\omega\|_{L^p}$  for $\omega_0\in C_c^\infty(\mathbb{R}^3)$ for  3D Euler equation in the whole space, even without restricting to the axisymmetric setting.

\medskip
\noindent
\textbf{Almost no vorticity is left behind in any bounded region.} Note that although \eqref{thm:main3} gives the $L^p$-norm growth of vorticity, the growth of norm are likely driven by the ``head'' part of the vorticity located near the outmost part of the support (e.g., the part in the orange rectangle on the right of Figure \ref{fig_dipole}), and it does not give much information on the ``tail'' part of the vorticity that is left behind.

Under the symmetry and sign assumptions \eqref{w0_odd}, since $\omega(\cdot,t)\geq 0$ on $\Pi^+$ for all times,  and the ``total vorticity mass'' on $\Pi^+$, given by $\iint_{\Pi^+} \omega(r,z,t) drdz$, is conserved in time, it is natural to ask where the vorticity is located as time evolves. Along the proof of Theorem~\ref{thm:main3}, we obtain a by-product that sheds light on this question. Heuristically speaking, for any bounded region $\Pi^+\cap\{r\le R\}$, we prove that as $t\to\infty$, almost no vorticity remains in this region in the time-integral sense. Since $R$ is arbitrary, this means that, in the time-integral sense, nearly all vorticity must eventually escape to $r\to\infty$. The precise statement is as follows:

\begin{theorem}[Nearly all vorticity escapes to $r\to\infty$ in the time-integral sense] \label{thm:main4}
Under the assumptions \eqref{assu_w0}, \eqref{w0_odd}, and $Z(0)<\infty$, let $\omega(\cdot,t)$ be the corresponding global-in-time solution of \eqref{eq:Euler2}. Then there exists some constants $C(\omega_0,u_0)>0$ and $R_0(\omega_0,u_0)\ge 1$ that only depends on the initial data, such that the following holds: for any $R\geq R_0$, defining the ``vorticity mass'' in $\Pi^{+}\cap \{r\le R\}$ as
\[
m_R(t) := \iint_{\Pi^{+} \cap \{r\le R\}} \omega(r,z,t) \,drdz,
\]
we have that $m_R(t)$ satisfies
\begin{equation} 
\int_0^\infty m_R(t)^4 dt \leq C(\omega_0,u_0) R^4 <\infty.
\end{equation}
\end{theorem}

\subsection{Overview of the proof strategy}
 In the proof of the upper bounds in Theorem~\ref{thm:main1}, the key step is to control the evolution of the second moment $P_2(t)$ through an elementary but useful observation: its time derivative can be bounded by an interpolation of kinetic energy, $\|\omega/r\|_{L^\infty}$ and $P_2$ itself. Once this is done, the bounds on $\|u_r\|_{L^\infty}$ and $\|\omega\|_{L^\infty}$ directly follows from applying the $P_2$ bound into an interpolation inequality in Lim--Jeong \cite[Prop. A.1]{LJ}.  Finally, for $k\geq 2$, we obtain upper bounds for $P_k(t)$ by bounding its time derivative using an interpolation among $\|u_r\|_{L^\infty}$, the kinetic energy, $\|\omega/r\|_{L^\infty}$ and $P_k$ itself.

\smallskip
In the proofs of Theorems~\ref{thm:main2}, \ref{thm:main3} and  \ref{thm:main4}, the common theme is to express the time derivative of $P_2(t)$ and $Z(t)$ from a new perspective. Previously, \cite{Choi-Jeong, GMT} have already shown that $P_2'(t)>0$ and $Z'(t)<0$, by expressing their time derivatives as double integrals of the form 
\begin{equation}\label{temp_kernel}
\iint_{\Pi^+}\iint_{\Pi^+} K(r,\overline{r},z,\overline{z}) \,\omega(r,z,t) \omega(\overline{r},\overline{z},t) \, drdz d\overline{r} d\overline{z}
\end{equation} for some  kernel $K$ that is positive in $P_2'(t)$, and negative in $Z'(t)$; in both cases, the expression for the kernel $K$ is quite involved. In this paper we take a completely different approach, and derive new representations for $P_2'(t)$ and $Z'(t)$ using the velocity field $u$, instead of $\omega$ (See Section~\ref{sec41} for the derivation):
\begin{align}\label{der_P}
P_2'(t) &= \int_{0}^{\infty} ru_r(r,0,t)^2 \, dr,\\
\label{der_Z}
-\!Z' (t) &= \int_{0}^{\infty} \frac{u_z(0,z,t)^2}{2} \, dz + \iint_{\Pi^+} \frac{u_r(r,z,t)^2}{r} \, drdz.
\end{align}
We find that when expressed in this form, the integrals on the right hand side are easier to work with, compared to the double integral representation \eqref{temp_kernel}. 

In particular, in the proof of Theorem~\ref{thm:main2}, in order to control $P'(t)$ from below, we apply the Cauchy-Schwarz inequality to the new formulation \eqref{der_P} to obtain \[
P_2'(t) \geq \int_{a(t)}^{b(t)} ru_r(r,0,t)^2 \, dr \ge \left(\int_{a(t)}^{b(t)} u_r(r,0,t) \, dr \right)^2 \left(\int_{a(t)}^{b(t)} r^{-1}dr\right)^{-1}
\]
for all $0<a(t)< b(t)$.
By choosing $b(t)\sim 1+t$, $a(t)\sim (1+t)^{-1/3}$, we are able to show that $\int_{a(t)}^{b(t)} u_r(r,0,t) \, dr$ is bounded below by a positive constant (see Section~\ref{sec42} for details), thus leading to $P_2'(t) \gtrsim (1+\log t)^{-1}$. 

In the proof of Theorems~\ref{thm:main3} and \ref{thm:main4}, the crucial step is to obtain some quantitative lower bound for $-Z'(t)$. Using the new formulation \eqref{der_Z}, we prove that for any sufficiently large $R$ and any positive $m\le m_0$, if the ``vorticity mass'' in the region $\Pi^+\cap \{r\le R\}$ is at least $m$ at some $t$, i.e.,
$
\iint_{\Pi^+\cap\{r\le R\}} \omega(r,z,t) \,drdz \ge m,
$
then we have the quantitative lower bound
\begin{equation}
\label{derZ_temp}
-Z'(t)\geq c(\omega_0,u_0) m^4 R^{-4}
\end{equation} at this time. (See Proposition~\ref{prop:new} for the precise statement.) Roughly speaking, it says that if there is a certain amount of vorticity in the bounded region $\{r\le R\}$ at some time, then $-Z'(t)$ cannot be too small at this time. 

On the other hand, since  $Z(t)$ is monotone decreasing and positive, $-Z'(t)$ is integrable in time over $[0,\infty)$, thus $-Z'(t)$ must decay to zero as $t\to\infty$ in the time-integral sense. We can then obtain Theorem~\ref{thm:main3} by fixing a suitable $m$ only depending on initial data, setting $R\sim t^{1/4}$ and integrate \eqref{derZ_temp} in time. Similarly, we obtain Theorem~\ref{thm:main4} by fixing $R$, choosing $m$ as $m_R(t)=\iint_{r\le R} \omega drdz$, and integrating in time. See Section~\ref{sec5} for the detailed proofs.

\subsection{Organization of the paper} 
In Section~\ref{sec2}, we collect various properties on the axi-symmetric Biot-Savart law and the energy. We then prove Theorem~\ref{thm:main1} for the upper bounds in Section~\ref{sec3}. In Section~\ref{sec4}--\ref{sec5}, we aim to derive various lower bounds on the solutions under the extra assumptions that $\omega_0$ is odd in $z$ and non-negative in $\{z>0\}$. Section~\ref{sec4} is devoted to the proof of Theorem~\ref{thm:main2}, which gives a lower bound of $t^{1-}$ growth rate of the radial moment $P_2$. Finally, in Section~\ref{sec5} we prove Theorems~\ref{thm:main3}--\ref{thm:main4}, which deal with growth of $\|\omega(\cdot,t)\|_{L^p(\mathbb{R}^3)}$, and show that nearly all vorticity escapes to $r\to\infty$ in the time-integral sense.

\ackn{YY has been supported by the MOE Tier 1 grant and the Asian Young Scientist Fellowship. The authors are very grateful to Tarek Elgindi and In-Jee Jeong for helpful discussions.}

\section{Preliminaries} \label{sec2}
\subsection{Notation and conventions}
Throughout this paper, we frequently use the following notation and conventions:
\begin{itemize}
    \item By the $rz$-plane we mean the half-plane \[ \Pi:= \{ (r,z) \, : \, r\ge 0, \, \, z\in \mathbb{R} \}. \] Its upper quadrant is denoted by \[ \Pi^{+}:= \{(r,z) \in \Pi \, : \, z>0 \}. \]

\smallskip
    \item Since the evolution of the vorticity vector $\omega$ is completely determined by the scalar component $\omega^{\theta}$, from now on, we will denote $\omega^\theta$ as $\omega$ (and the initial data $\omega^\theta_0$ as $\omega_0$), and refer to it as the vorticity throughout this paper.

    \smallskip
    \item We use the abbreviation $\omega$ for $\omega (r,z,t)$ and $\overline{\omega}$ for $\omega (\overline{r}, \overline{z},t)$ in integrals and formulas, whenever the dependence is clear from the context. Similarly, we write $u_r$ and $u_z$ for $u_r(r,z,t)$ and $u_z(r,z,t)$, while $\overline{u}_r$ and $\overline{u}_z$ are used for $u_r(\overline{r},\overline{z},t)$ and $u_z(\overline{r},\overline{z},t)$, respectively.

\smallskip
    \item The measure $\mu$ is defined in $\Pi$ by $d\mu:= r \, drdz$. More precisely, for $A\subset \Pi$ we have \[ \mu (A) := \iint_{A} r\, drdz. \]
\end{itemize}

\subsection{Axisymmetric Biot--Savart law}	
We begin by recalling the axisymmetric Biot--Savart law and some useful facts in the standard 3D setting. In the axisymmetric without swirl case, the stream function is recovered from the vorticity by
\begin{equation*}
\psi (r,z,t) = \frac{1}{2\pi} \iint_{\Pi} \sqrt{r\overline{r}} \, F\left( \frac{(r-\overline{r})^2 +(z-\overline{z})^2}{r\overline{r}} \right) \, \overline{\omega} \, d\overline{r}d\overline{z}, 
\end{equation*}
where $F$ is an elliptic integral given by 
\begin{equation} \label{eq:bs_F}
F(s) = \int_{0}^{\pi} \frac{\cos \theta}{(2(1-\cos \theta) + s)^{1/2}} \, d\theta, \quad \quad s>0.
\end{equation}
The velocity field becomes 
\[
u = \frac{1}{r}\nabla^{\perp} \psi, \quad \quad \nabla^{\perp}:= (\partial_z, -\partial_r)\footnote{The usual notation is $\nabla^{\perp}:= (-\partial_z, \partial_r)$, however, in this case, we would need to define the vorticity as negative in the upper quadrant.}
\]
and by denoting \[ D=D(r,z,\overline{r},\overline{z}):= \frac{(r-\overline{r})^2 +(z-\overline{z})^2}{r\overline{r}}, \] we can write the radial and vertical velocity components as
\begin{equation} \label{u_r}
u_r(r,z,t) = \frac{1}{\pi r} \iint_{\Pi} \frac{(z-\overline{z})}{\sqrt{r\overline{r}}} \, F'(D) \, \overline{\omega} \, d\overline{r}d\overline{z},
\end{equation}
\begin{equation} \label{u_z}
u_z(r,z,t) = -\frac{1}{\pi r} \iint_{\Pi} \sqrt{r\overline{r}}\left[\frac{1}{4r} F(D) + \left(\frac{r-\overline{r}}{r\overline{r}} - \frac{D}{2r} \right) F'(D) \right] \, \overline{\omega} \, d\overline{r}d\overline{z}.
\end{equation}

Next, we recall some properties of $F$, namely, its asymptotic behavior near $0$ and $\infty$. It is not difficult to observe that $F$ is strictly positive and decreasing for $s>0$, diverging near $0$ and vanishing as $s$ tends to infinity. Moreover, we have the following asymptotics estimates: 
\begin{align*} %\label{eq:bs_F1}
    F(s) &= \frac{1}{2} \log \frac{1}{s} + \log 8 -2 + O(s\log s), \quad s\rightarrow 0^{+}, \\
    F(s) &= \frac{\pi}{2} \frac{1}{s^{3/2}} + O(s^{-5/2}), \quad s\rightarrow +\infty.    
\end{align*}
Similarly, for $F'(s)$, we observe that it is strictly negative and increasing for $s>0$, with the following asymptotics near the end points: 
\begin{align*} %\label{eq:bs_F2}
    -F'(s) &= \frac{1}{2s} + O(\log s), \quad   s\rightarrow 0^{+}, \\
    -F'(s) &= \frac{3\pi}{4} \frac{1}{s^{5/2}} + O (s^{-7/2}), \quad s\rightarrow +\infty.    
\end{align*}
Continuing this, one can also derive such asymptotics for the $k$-th derivative of $F$. 
We refer the readers to \cite{FeSv} for the proof of these estimates, along with other properties of $F$ and its derivatives.

\subsection{The kinetic energy} 
Recall that the kinetic energy 
\[ E(t)= \frac{1}{2} \iiint_{\mathbb{R}^3} |u(x,t)|^2 \, dx \] 
is conserved in time and finite by our assumptions on the initial data. In the axisymmetric without swirl setting, this becomes 
\[
E(t) = \pi \iint_{\Pi} r(u_r^2+u_z^2) \, drdz,
\]
showing that it has nonnegative contributions from both components. Defining $E_0:= \frac{1}{\pi} E(t)$, we have 
\[
E_0 = \iint_{\Pi} r(u_r^2+u_z^2) \, drdz,
\]
and it can be treated as a positive constant depending only on the initial data.

\section{Upper bounds of the radial moments and $L^\infty$ norm}	
\label{sec3}

The goal of this section is to prove Theorem~\ref{thm:main1}. Before the proof, we state a technical lemma that controls the evolution of the radial moments $P_k := \iint_{\Pi} r^k |\omega| drdz$ in time. Note that if the integral is defined as $\iint_{\Pi} r^k \omega drdz$ (without the absolute value on $\omega$), its time derivative can be directly computed as $ k\iint_{\Pi} r^{k-1}  u_r \omega \, drdz$ (see \cite[Eq.(2.8)]{GMT}). With the absolute value, it is slightly more involved to control $P_k'(t)$, where we overcome the lack of regularity of the absolute value function by approximating it with a sequence of smooth functions.

\begin{lemma}\label{lem_pk_der}
Under the assumptions of Theorem \ref{thm:main1}, for each $k\geq 2$, we have the following differential inequality for $P_k$:
\begin{equation}\label{diff_ineq_pk}
P_k'(t) \leq k\iint_{\Pi} r^{k-1}  |u_r \omega| \, drdz.
\end{equation}
\end{lemma}

\begin{proof}
To rigorously justify \eqref{diff_ineq_pk}, we will take a regularization of the absolute value function $f(\omega) := |\omega|$ in $P_k$. For any $\epsilon>0$, we approximate $f$ by a smooth function $f_\epsilon\in C^\infty(\mathbb{R})$ defined as
\[
f_\epsilon(s) := \sqrt{s^2 + \epsilon^2} - \epsilon.
\]
For any $s\in \mathbb{R}$, we clearly have $0\leq f_{\epsilon}(s) \leq s$, $|f_\epsilon'(s)|<1$, and $\lim_{\epsilon \searrow 0} f_\epsilon(s)=|s|$.
Using such function $f_\epsilon(\omega)$ to replace $|\omega|$ in $P_k$, we introduce
\[
P_k^\epsilon(t) := \iint_{\Pi} r^k  f_\epsilon(\omega(r,z,t)) \, drdz,
\]
and compute its time derivative as follows:
\[
\begin{split}
\frac{d}{dt} P_k^\epsilon(t) &=  \iint_{\Pi} r^{k} f_\epsilon'(\omega) \partial_t \omega \, drdz =  \iint_{\Pi} r^{k+1} f_\epsilon'(\omega) \partial_t \Big( \frac{\omega}{r}\Big) \, drdz \\
&=  -\iint_{\Pi} r^{k+1} f_\epsilon'(\omega)  u \cdot \nabla \Big(\frac{\omega}{r}\Big)  \, drdz = -\iint_{\Pi} r^{k} f_\epsilon'(\omega) \,  (r u) \cdot \nabla \Big(\frac{\omega}{r}\Big)  \, drdz,
\end{split}
\]
where we denote $u := (u_r, u_z)$, and $\nabla := (\partial_r, \partial_z)$ as the gradient in the $rz$-plane. Applying divergence theorem to the last integral, and using the fact that  $\nabla\cdot (ru)=0$, we have
\begin{equation}\label{pk_temp}
\begin{split}
\frac{d}{dt} P_k^\epsilon(t) &= \iint_{\Pi} \nabla (r^{k} f_\epsilon'(\omega)) \cdot  u \omega \, drdz \\
&=  \iint_{\Pi} k r^{k-1} f_\epsilon'(\omega)   u_r \omega \, drdz + \iint_{\Pi} r^k f_\epsilon''(\omega) \nabla\omega \cdot u \omega \, drdz =: I_1 + I_2.
\end{split}
\end{equation}
Introducing 
\[
g_\epsilon(s) := \int_0^s f_\epsilon''(y) y \,dy,
\] we can rewrite $I_2$ as
\[
\begin{split}
I_2 &= \iint_{\Pi} r^k \underbrace{f_\epsilon''(\omega) \omega}_{=g_\epsilon'(\omega)} \nabla\omega \cdot u ~drdz =  \iint_{\Pi} r^k \nabla (g_\epsilon(\omega) ) \cdot u ~drdz\\
&=    -\iint_{\Pi} (k-1)r^{k-1} g_\epsilon(\omega)  u_r  ~drdz.\\
\end{split}
\]
Plugging it into \eqref{pk_temp} and integrating in time, the following holds for any $t_2>t_1\geq 0$:
\[
P_k^\epsilon(t_2) - P_k^\epsilon(t_1) \leq \int_{t_1}^{t_2} \iint_{\Pi} k r^{k-1} |f_\epsilon'(\omega)   u_r \omega| \, drdz dt + \int_{t_1}^{t_2} \iint_{\Pi} (k-1) r^{k-1} |g_\epsilon(\omega)   u_r| \, drdz dt.
\]
Since $g_\epsilon(s) = f_\epsilon'(s) s - f_\epsilon(s) = \epsilon-\frac{\epsilon^2}{s^2+\epsilon^2}$, one can check that $|g_\epsilon(s)| \leq \min\{\epsilon,|s|\}$. Sending $\epsilon\searrow 0$ and applying dominating convergence theorem, we arrive at
\[
P_k(t_2)-P_k(t_1) \leq  \int_{t_1}^{t_2} \iint_{\Pi} k r^{k-1} |  u_r \omega| \, drdz dt \quad\text{ for all } t_2>t_1\geq 0,
\]
finishing the proof.
\end{proof}

Now we are ready to prove Theorem ~\ref{thm:main1}.

\begin{proof}[Proof of Theorem \ref{thm:main1}]
We divide the proof into the following steps:

\medskip
\noindent \textbf{Step 1: Upper bound for $P_2$}.
Let us first obtain the bound \eqref{upper_bd_pk} for the special case $k=2$. By Lemma \ref{lem_pk_der}, we have
\begin{equation}\label{p2temp}
P_2'(t) \leq 2\iint_{\Pi} r |u_r \omega| \, drdz = 2\iint_{\Pi} r^2 |u_r|  \frac{| \omega|}{r} \, drdz.
\end{equation}
Denote $A_0 := \|\omega_0/r\|_{L^{\infty}}$, which is finite by our assumption on $\omega_0$.  Since $\|\omega(\cdot,t)/r\|_{L^{\infty}}$ is conserved in time, we have $\frac{| \omega|}{r} \leq A_0^{1/2} (\frac{| \omega|}{r})^{1/2}$. Plugging it into \eqref{p2temp} yields
\[
P_2'(t) \leq 2\iint_{\Pi} r^2 |u_r| A_0^{1/2} \Big(\frac{| \omega|}{r}\Big)^{1/2} \, drdz = 2 A_0^{1/2} \iint_{\Pi} r^{3/2} |u_r| |\omega|^{1/2}\, drdz.
\]
We then apply Cauchy--Schwarz inequality to obtain
\[
\begin{split}
P_2'(t) &\leq 2A_0^{1/2} \left(\iint_{\Pi} r|u_r|^2 \,drdz\right)^{1/2} \left( \iint_{\Pi} r^2 |\omega| \, drdz \right)^{1/2}\\
&\leq 2A_0^{1/2} E_0^{1/2} P_2^{1/2},
\end{split}
\]
where we used the conservation of the kinetic energy, which is finite by assumption on the initial data, and $E_0:= \frac{1}{\pi} E(t) = \iint_{\Pi} r(|u_r|^2 + |u_z|^2) \,drdz$, as denoted earlier.

Solving this differential inequality, we have the following quadratic upper bound for $P_2(t)$:
\begin{equation}\label{p2_temp}
P_2(t) \le \left(P_2(0)^{1/2}+(A_0 E_0)^{1/2} t\right)^2.
\end{equation}

\medskip
\noindent \textbf{Step 2: Upper bounds for $\|u_r\|_{L^\infty}$ and $\|\omega\|_{L^\infty}$}.
Next we move on to prove \eqref{upper_bd_u} and \eqref{upper_bd_omega}. (The proof of \eqref{upper_bd_u} for $k>2$ will be postponed till the end, since its proof relies on \eqref{upper_bd_u}.) Using the bound \eqref{p2_temp} that we just obtained, \eqref{upper_bd_u} is a direct consequence when we apply the following interpolation inequality in Lim--Jeong \cite[Prop. A.1]{LJ} (the original inequality is written using norms in $\mathbb{R}^3$, and we rewrite it in the 2-dimensional domain $\Pi$):
\[
\|u_r\|_{L^\infty(\Pi)} \leq C\|\sqrt{r}u\|_{L^2(\Pi)}^{1/3} \Big\|\frac{\omega}{r}\Big\|_{L^\infty(\Pi)}^{1/2} \|r^2\omega\|_{L^1(\Pi)}^{1/6} = C E_0^{1/3} A_0^{1/2} P_2(t)^{1/6},
\]
where $C>0$ is a universal constant.
Plugging \eqref{p2_temp} into above immediately yields 
$\|u_r(\cdot,t)\|_{L^\infty(\mathbb{R}^3)} \leq C_0 (1+t)^{1/3},
$
where $C_0$ depends on $A_0, E_0$ and $P_2(0)$. This gives \eqref{upper_bd_u}. 

Once we have \eqref{upper_bd_u}, the upper bound for vorticity \eqref{upper_bd_omega} directly follows the conservation of $\omega/r$ along flow maps, which we briefly sketch below for the sake of completeness. For a fluid particle originates from $(r,z)$ with $r>0$, denote its flow map at time $t$ by $\phi_t(r,z)$, whose time evolution satisfies $\frac{d}{dt} \phi_t(r,z) = u(\phi_t(r,z),t)$. Applying \eqref{upper_bd_u}, the $r$-component of $\phi_t(r,z)$ (which we denote by $\phi_t^r(r,z)$) satisfies
\[
0<\phi_t^r(r,z) \leq r + \int_0^t \|u_r(s)\|_{L^\infty} ds \leq r + C_1 (1+t)^{4/3},
\] 
where $C_1$ depends on $A_0, E_0$ and $P_2(0)$.
As a result, for any $(r,z)\in\Pi$ with $r>0$, we have
\begin{equation} \label{temp_using}
|\omega(\phi_t(r,z),t)| = \frac{|\omega_0(r,z)|}{r} \phi_t^r(r,z) \leq \frac{|\omega_0(r,z)|}{r} (r+C_1(1+t)^{4/3}).    
\end{equation}
We can then discuss two cases: 
\begin{itemize} 
\item If $r<C_1(1+t)^{4/3}$, we have 
\[
|\omega(\phi_t(r,z),t)| \leq \frac{|\omega_0(r,z)|}{r} \cdot 2C_1(1+t)^{4/3} \leq 2C_1 A_0 (1+t)^{4/3}.
\]
\item If $r\geq C_1(1+t)^{4/3}$, we simply have $|\omega(\phi_t(r,z),t)|\leq 2|\omega_0(r,z)| \leq 2\|\omega_0\|_{L^\infty}$.
\end{itemize}
Combining the above two cases gives 
\[
\|\omega(\cdot,t)\|_{L^\infty} \leq C_2 (1+t)^{4/3},
\]
where $C_2$ depends on $A_0, E_0, P_2(0)$ and $\|\omega_0\|_{L^\infty}$. This finishes the proof of \eqref{upper_bd_omega}.

\medskip
\noindent \textbf{Step 3: Upper bounds for $P_k$ with $k>2$}.
For $k>2$, we again apply Lemma \ref{lem_pk_der} to obtain
\[
P_k'(t) \leq k\iint_{\Pi} r^{k-1}  |u_r \omega| \, drdz = k\iint_{\Pi} r^{k}  |u_r|  \Big|\frac{\omega}{r}\Big| \, drdz
\]
Again, using that $A_0 := \|\omega_0/r\|_{L^{\infty}} = \|\omega(t)/r\|_{L^{\infty}} $, we have
$|\frac{\omega}{r}| \leq A_0^{1-b} |\frac{\omega}{r}|^b $ for any $0<b\leq 1$, where we choose $b=\frac{k-1}{k}$. Plugging it into above gives
\begin{align*}
P_k'(t) &\leq k A_0^{1/k} \iint_{\Pi} r^{k}  |u_r|  \Big|\frac{\omega}{r}\Big|^{(k-1)/k} \, drdz \\
&= k A_0^{1/k}  \iint_{\Pi} r^{1/k}  |u_r| \cdot r^{k-1} |\omega|^{(k-1)/k} \, drdz\\
&\leq k A_0^{1/k} \|u_r\|_{L^\infty}^{(k-2)/k} \iint_{\Pi} r^{1/k}  |u_r|^{2/k} \cdot r^{k-1} |\omega|^{(k-1)/k} \, drdz\\
&\leq k A_0^{1/k} \|u_r\|_{L^\infty}^{(k-2)/k} \left\|r^{1/k}  |u_r|^{2/k}\right\|_{L^k(\Pi)} \left\|r^{k-1} |\omega|^{(k-1)/k}\right\|_{L^{k/(k-1)}(\Pi)}\\
&\leq k A_0^{1/k} \|u_r(t)\|_{L^\infty}^{(k-2)/k} E_0^{\frac{1}{k}} P_k(t)^{\frac{k-1}{k}}.
\end{align*}

Finally, plugging the upper bound \eqref{upper_bd_u} for $\|u_r\|_{L^\infty}$ into above gives
\[
\begin{split}
P_k'(t) &\leq k A_0^{1/k} C_0^{(k-2)/k}  E_0^{1/k} (1+t)^{(k-2)/3k} P_k(t)^{(k-1)/k}\\ 
&\leq k C_3  (1+t)^{(k-2)/3k} P_k(t)^{(k-1)/k}.
\end{split}
\]
Here $C_3$ only depends on $A_0, E_0, P_2(0)$ (and is independent of $k$, since the powers $1/k$ and $(k-2)/k$ are both between 0 and 1).  Solving this differential inequality gives
\[
P_k(t) \leq \left(P_k(0)^{\frac{1}{k}} + C_3 (1+t)^{\frac{4}{3}-\frac{2}{3k}} \right)^k \leq C_4(k, u_0, \omega_0) (1+t)^{\frac{4}{3}k-\frac{2}{3}}, 
\]
finishing the proof of \eqref{upper_bd_pk} for the $k>2$ cases.
\end{proof}

\section{$t^{1-}$ lower bound for the radial moment $P_2$}
\label{sec4}
In this section, we prove Theorem \ref{thm:main2}, which gives a $t^{1-}$ lower bound for the $r^2$-radial moment $P_2$. Recall that the initial vorticity $\omega_0$ is assumed to be odd-in-$z$ and nonnegative in the upper quadrant $\Pi^+$. These assumptions allow us to redefine our radial and vertical moments as 
\begin{equation} \label{eq:radial_vertical}
P_k(t) := \iint_{\Pi^+} r^k \omega\, drdz, \quad \quad Z(t) := \iint_{\Pi^+} z\omega \, drdz.
\end{equation}
From \eqref{eq:radial} and \eqref{eq:vertical}, it follows that the radial and vertical moments over the whole $\Pi$ are simply twice those over $\Pi^+$. Therefore, we adopt the definitions  \eqref{eq:radial_vertical} throughout this section. We also assume finiteness of the initial moments, 
\begin{equation} \label{eq:in_moments}
0 < P_2(0) < \infty, \quad 0 < Z(0) < \infty.
\end{equation}

Before moving to our analysis, we recall the following monotonicity properties of the moments established in Choi--Jeong \cite{Choi-Jeong}:

\begin{lemma}[Theorem 1.1 of \cite{Choi-Jeong}] \label{lem:monotonicity}
Under the assumptions \eqref{assu_w0}, \eqref{w0_odd} and \eqref{eq:in_moments}, the moments $P_2$ and $Z$ defined in \eqref{eq:radial_vertical} satisfy that 
$P_2(t)$ is strictly increasing in time, while $Z(t)$ is strictly decreasing in time. 
\end{lemma}

This monotonicity result played a crucial role in the proofs of both Choi--Jeong \cite{Choi-Jeong} and Gustafson--Miller--Tsai \cite{GMT}, where both papers expressed $P_2'(t)$ and $Z'(t)$ as double integrals of the form 
\[
\iint_{\Pi^+}\iint_{\Pi^+} K(r,\overline{r},z,\overline{z}) \,\omega\overline{\omega} \, drdz d\overline{r} d\overline{z}
\] for some  kernel $K$ that is positive in $P_2'(t)$, and negative in $Z'(t)$. The main difference of our approach from those in  is to derive new representations for the time derivative of the moments. As we will see, we will find alternative (but identical) expressions that express these integrals using the velocity field $u$, instead of $\omega$. Moreover, we rely on the vorticity total mass, whereas in the previous works the kinetic energy plays a crucial role.

\subsection{A new representation for the time derivatives of moments}\label{sec41}
Under the odd-in-$z$ assumption, the upper quadrant $\Pi^+$ remains invariant, meaning that particle trajectories never cross $z=0$. So $u_z$ vanishes at $z=0$ (i.e., on the $r$-axis). Moreover, from the axisymmetric no-swirl structure, we know that $u_r$ vanishes along the $z$-axis. With these boundary conditions, one can simply integrate by parts to find
\begin{equation*}
P_2'(t) = \iint_{\Pi^+} r^2 \partial_t \omega \, drdz = - \iint_{\Pi^+} r^2 \, \nabla \cdot (u\omega) \, drdz = 2\iint_{\Pi^+} ru_r \omega \, drdz,
\end{equation*}
where, throughout this section, $\nabla:=(\partial_r,\partial_z)$. Substituting $\omega = -\partial_z u_r + \partial_r u_z$ and re-applying integration by parts yields
\begin{align} \label{radial_1}
P_2'(t) = \int_{0}^{\infty} ru_r(r,0,t)^2 \, dr.
\end{align}
Similarly, for the time derivative of $Z(t)$ (which we will use in the next section), we have
\[
Z'(t) = \iint_{\Pi^+} z \, \partial_t \omega \, drdz = - \iint_{\Pi^+} z \, \nabla \cdot (u\omega) \, drdz = \iint_{\Pi^+} u_z  \omega \, drdz,
\] 
and
\begin{equation} \label{vertical_1}
-\!Z' (t) = \int_{0}^{\infty} \frac{u_z(0,z,t)^2}{2} \, dz + \iint_{\Pi^+} \frac{u_r(r,z,t)^2}{r} \, drdz.
\end{equation}

From the axisymmetric Biot--Savart law, expressions \eqref{u_r} and \eqref{u_z}, one can evaluate $u_r(r,0,t)$ as
\begin{align} \nonumber
u_r(r,0,t) &= \frac{1}{\pi} \iint_{\Pi} \frac{\overline{z}}{r\sqrt{r\overline{r}}} \, \left(-F'(D)\right) \, \overline{\omega} \, d\overline{r}d\overline{z}\\
&= \frac{2}{\pi} \iint_{\Pi^+} \frac{\overline{z}}{r\sqrt{r\overline{r}}} \, \left(-F'\left(\frac{(r-\overline{r})^2+\overline{z}^2}{r\overline{r}}\right)\right) \, \overline{\omega} \, d\overline{r}d\overline{z},
\label{ur}
\end{align}
 and $u_z(0,z,t)$, using the asymptotic estimates for $F$ and $F'$ at $\infty$, as
\begin{align} \nonumber
-u_z(0,z,t) &= \frac{1}{\pi} \lim_{r\to 0^{+}} \iint_{\Pi} \frac{\sqrt{\overline{r}}}{\sqrt{r}} \left[ \frac{1}{4r} F\left(D\right) + \left( \frac{r-\overline{r}}{r\overline{r}} - \frac{D}{2r} \right)F'(D) \right] \overline{\omega} \, d\overline{r}d\overline{z} \\
\nonumber
&= \frac{1}{\pi} \lim_{r\to 0^{+}} \iint_{\Pi} \left[ \frac{\pi}{8} \frac{\overline{r}^{2}}{\left(\overline{r}^2+(z-\overline{z})^2\right)^{3/2}} - \frac{3\pi}{4} \left( \frac{\overline{r}^{1/2} (r-\overline{r})}{r^{3/2} \overline{r}} - \frac{\overline{r}^2+(z-\overline{z})^2}{2r^{5/2}\overline{r}^{1/2}} \right) \frac{(r\overline{r})^{5/2}}{(\overline{r}^2+(z-\overline{z})^2)^{5/2}} \right] \overline{\omega} \, d\overline{r}d\overline{z} \\
\nonumber
&= \frac{1}{2} \iint_{\Pi} \frac{\overline{r}^2}{(\overline{r}^2+(z-\overline{z})^2)^{3/2}} \, \overline{\omega} \, d\overline{r}d\overline{z} \\
&= \frac{1}{2} \iint_{\Pi^+} \left[ \frac{\overline{r}^2}{(\overline{r}^2+(z-\overline{z})^2)^{3/2}} - \frac{\overline{r}^2}{(\overline{r}^2+(z+\overline{z})^2)^{3/2}} \right] \overline{\omega} \, d\overline{r}d\overline{z}.
\label{uz}
\end{align}
Thus, for each $t\ge 0$ we have: $u_r(r,0,t)>0$ for all $r\in \mathbb{R}^{+}$ and $-u_z(0,z,t)>0$ for all $z\in \mathbb{R}$. 
Together with the representations \eqref{radial_1} and \eqref{vertical_1}, this immediately implies that $P_2'(t) >0$ and $-Z'(t)>0$ for all $t\ge 0$. In other words, we provide an alternative proof of Lemma \ref{lem:monotonicity}, showing that $P_2(t)$ is strictly increasing and $Z(t)$ is strictly decreasing in $t\ge 0$.

\subsection{The total vorticity mass} 
\label{sec42}
Since $\|\omega/r\|_{L^p(\mathbb{R}^3)}$ is conserved for all $p\geq 1$, by rewriting the $L^1$ norm as integral on the $rz$-plane (and using the odd-in-$z$ symmetry and non-negativity in $\Pi^+$), we know the following quantity  \[ m(t) = \iint_{\Pi^+} \omega(r,z,t) \, drdz \]
is conserved for all time, which we call the \emph{total vorticity mass}, and denote it by $m_0$. Plugging $\omega = -\partial_z u_r + \partial_r u_z$ into the above integral and performing integration by parts (where both $u_r$ and $u_z$ vanish at infinity), we obtain
\[
m_0 = \int_{0}^{\infty} u_r(r,0,t) \, dr + \int_{0}^{\infty} (-u_z(0,z,t)) \, dz.
\]
Since $u_r>0$ along the $r$-axis and $-u_z>0$ along the $z$-axis, we observe that both integrals above are positive. 

In fact, it is possible to evaluate these integrals explicitly. In particular, one has
\begin{align} 
\nonumber
\int_{0}^{\infty} \left(-u_z(0,z,t)\right) \, dz &= \frac{1}{2} \iint_{\Pi^+} \overline{r}^2 \overline{\omega} \left[ \int_{0}^{\infty} \left(\frac{1}{(\overline{r}^2+(z-\overline{z})^2)^{3/2}} - \frac{1}{(\overline{r}^2+(z+\overline{z})^2)^{3/2}}  \right) \, dz \right] \, d\overline{r}d\overline{z}\\
\nonumber
&= \frac{1}{2} \iint_{\Pi^+} \overline{r}^2 \overline{\omega} \left[ \int_{-\overline{z}}^{\overline{z}} \frac{1}{(\overline{r}^2+z^2)^{3/2}} \,dz \right] \, d\overline{r}d\overline{z} \\
\label{mass1}
&= \iint_{\Pi^+} \frac{\overline{z}}{\sqrt{\overline{r}^2+\overline{z}^2}} \, \overline{\omega} \, d\overline{r}d\overline{z} = \iint_{\Pi^+} \frac{z}{\sqrt{r^2+z^2}} \, \omega \, drdz,
\end{align}
which is a portion of the total mass as $0< \frac{z}{\sqrt{r^2+z^2}} <1$ in $\Pi^{+}$. Consequently,
\begin{align} \label{mass2}
\int_{0}^{\infty} u_r(r,0,t)\, dr = \iint_{\Pi^+} \left(1- \frac{z}{\sqrt{r^2+z^2}}\right)  \omega \, drdz.
\end{align}
Hence, comparing the integrand factors in \eqref{mass1} and \eqref{mass2} reveals which axis integral contributes more to the total mass, particularly in certain regions.

Now we are ready to prove Theorem \ref{thm:main2}. 

\begin{proof}[Proof of Theorem \ref{thm:main2}]
At each $t\geq 0$, for any $0<a<b<\infty$, by Cauchy-Schwarz inequality we obtain 
\[
 \left( \int_{a}^{b} r u_r(r,0,t)^2 \, dr \right) \left(\int_{a}^{b} \frac{1}{r} \, dr \right) \ge \left(\int_{a}^{b} u_r(r,0,t)\, dr\right)^2.
\]
Noting \eqref{radial_1}, we then have
\begin{equation} \label{p2_estimate}
P_2'(t) \ge \frac{\left(\int_{a}^{b} u_r(r,0,t) \, dr \right)^2}{\log (b/a)}.
\end{equation}
To obtain a lower bound on the right hand side (possibly depending on $t$), the rest of the proof will be about choosing $a$ and $b$, depending on $t$ as well as the initial data $\omega_0$ and $u_0$, so that the integral $\int_a^b u_r(r,0,t) \, dr$ admits a uniform positive lower bound for all time, while $\log(b/a)$ remains controlled from above.

First, we claim that 
\begin{equation}\label{int_lower}\int_0^\infty u_r(r,0,t) dr > \gamma(\omega_0)>0 \quad\text{ for all }t\geq 0,
\end{equation} where $\gamma(\omega_0)$ is a positive constant depending on $m_0 := \int_{\Pi^+} \omega_0 drdz$, $A_0 := \|\omega_0/r\|_{L^\infty}$ and $Z(0)$.
To show this, first note that for any $H>0$ and any $t>0$, the mass of $\omega(\cdot,t)$ in $\{z>H\}$ can be bounded by
\[
\iint_{\Pi^+} \omega(r,z,t) 1_{\{z>H\}} \,drdz \leq  \iint_{\Pi^+} \frac{z}{H} \omega(r,z,t)  \,drdz \leq \frac{Z(t)}{H} \leq \frac{Z(0)}{H}, 
\]
therefore by fixing 
\[H:= \frac{m_0}{2Z(0)},
\] we know the mass of $\omega(\cdot,t)$ in $\{z>H\}$ never exceeds $m_0/2$ for all $t\geq 0$.

Next let us consider the straight line $z=k r$ in $\Pi^+$, where the slope $k>0$ will be fixed momentarily. This line divides the region $\{ 0 < z < H\}$ into two parts. The left part $\{ 0 < z < H, 0<r<z/k\}$ is a triangle, whose maximum $r$-coordinate is $H/k$. Using the fact that $\|\omega/r\|_{L^\infty} = \|\omega_0/r\|_{L^\infty}=A_0 $, $\omega(\cdot,t)$ is uniformly bounded in this triangle for all times by 
\[
\omega(\cdot,t) \leq \left\|\frac{\omega}{r}\right\|_{L^\infty} \frac{H}{k} \leq \frac{A_0 H}{k} \quad \text{ for } 0<z<H, \,0<r<\frac{z}{k}.
\] Since the area of this triangle is $H^2/(2k)$, the mass of $\omega(\cdot,t)$ in this triangle can be bounded by
\[
\int_0^H \int_0^{z/k} \omega(r,z,t) \,drdz \leq \frac{A_0 H}{k} \frac{H^2}{2k}  = \frac{A_0 m_0^3}{16 k^2 Z(0)^3 },
\]
where we use the definition of $H$ in the last equality. Setting 
$k:= \frac{A_0^{1/2} m_0}{2 Z(0)^{3/2}},
$ the right hand side above becomes $m_0/4$. 

Combining the above mass upper bounds in $\{z>H\}$ and the triangle, we know that with the above definition of $H$ and $k$, the right part $\{ 0 < z < H, r>z/k\}$ contains at least mass $m_0/4$ for all times. Applying this to \eqref{mass2}, we obtain
\begin{align*}
\int_{0}^{\infty} u_r(r,0,t)\, dr &\ge  \int_0^H \int_{z/k}^\infty  \left(1- \frac{z}{\sqrt{r^2+z^2}}\right)\omega(r,z,t)  \, drdz\\
& > \left( 1-\frac{k}{\sqrt{1+k^2}}\right) \frac{m_0}{4} =: \gamma,
\end{align*}
finishing the proof of the claim \eqref{int_lower}.

On the other hand,  we will control $\int_0^a u_r(r,0,t) dr$ and $\int_b^\infty u_r(r,0,t) dr$ from above using the upper bounds for $\|u_r(t)\|_{L^\infty}$ and $P_2(t)$ established in Theorem~\ref{thm:main1}. From \eqref{upper_bd_u}, it follows that
\begin{equation}\label{int_a}
\int_{0}^{a} u_r(r,0,t) \, dr \le a\|u_r(\cdot,t)\|_{L^\infty(\mathbb{R}^3)} \le  a \, C_5 (1+t)^{1/3}, 
\end{equation}
where $C_5$ depends only on $\omega_0$ and $u_0$. Therefore, by choosing $a=a(t;\omega_0,u_0)$ as
\begin{equation}\label{def_a}
a=a(t;\omega_0,u_0):= \frac{\gamma}{4C_5(1+t)^{1/3}},
\end{equation}
we achieve the bound $\int_{0}^{a} u_r(r,0,t) \, dr \le \frac{\gamma}{4}$ for all $t\geq 0$.

To control $\int_b^\infty u_r(r,0,t) dr$, note that although $P_2(t) = \iint_{\Pi^+} \omega r^2 drdz$ is defined as a double integral in $\Pi^+$, it is identical to a line integral along the $r$-axis: plugging in $\omega = -\partial_z u_r + \partial_r u_z$ and performing integration by parts gives
\begin{equation}\label{p2_alt}
P_2(t) = \iint_{\Pi^+} (-\partial_z u_r + \partial_r u_z) r^2 \,drdz = \int_0^\infty r^2 u_r(r,0,t) dr,
\end{equation}
where we used that 
\[\iint_{\Pi^+} (\partial_r u_z) r^2 \,drdz = -\iint_{\Pi^+} 2r u_z \,drdz = \iint_{\Pi^+} 2 \partial_r\psi \,drdz = 0. 
\]
Combining \eqref{p2_alt} with the upper bound \eqref{upper_bd_pk} with $k=2$ yields 
\begin{equation} \label{int_b} 
\int_{b}^{\infty} u_r(r,0,t) dr \le \int_{b}^{\infty} \frac{r^2}{b^2} u_r(r,0,t) dr =  \frac{P_2(t)}{b^2} \le \frac{C_6(1+t)^2}{b^2},
\end{equation}
where $C_6$ depends only on $\omega_0$ and $u_0$. Choosing $b=b(t;\omega_0,u_0)$ as
\begin{align} \label{eq:a-b}
\begin{split}
b=b(t;\omega_0,u_0):=\frac{2C_6^{1/2} (1+t)}{\gamma^{1/2} },
\end{split}
\end{align}
we guarantee that $\int_{b}^{\infty} u_r(r,0,t) \leq \frac{\gamma}{4}$ for all $t\geq 0$. 

Therefore, combining \eqref{int_lower} with \eqref{int_a} and \eqref{int_b} gives
\begin{equation} \label{eq:5}
\int_{a(t;\omega_0,u_0)}^{b(t;\omega_0,u_0)} u_r(r,0,t) \, dr \geq \gamma - \int_{0}^{a(t;\omega_0,u_0)} u_r(r,0,t) \, dr - \int_{b(t;\omega_0,u_0)}^{\infty} u_r(r,0,t) \, dr \ge \frac{\gamma}{2}.
\end{equation}
Also, the  definitions of $a(t;\omega_0,u_0), b(t;\omega_0,u_0)$ in \eqref{def_a} and \eqref{eq:a-b} yields
\begin{equation}\label{temp_ab}
\log \frac{b(t;\omega_0,u_0)}{a(t;\omega_0,u_0)} \leq C(\omega_0,u_0)(1+\log(1+t)) \leq  C(\omega_0,u_0)\log(2+t)
\end{equation}
for some $C(\omega_0,u_0)>0$.

Finally, plugging \eqref{eq:5} and \eqref{temp_ab} into \eqref{p2_estimate}, we obtain
\begin{equation*}
P_2'(t) \ge \frac{c(\omega_0,u_0)}{\log (2+t)} \quad  \text{for all} \, \, t\ge 0,
\end{equation*} 
where $c(\omega_0,u_0)>0$ is a constant depending only on $\omega_0$ and $u_0$.
 Integrating it over $[0,t]$ gives
\begin{equation*}
P_2(t) \ge P_2(0) + \frac{c(\omega_0,u_0) \, t}{\log (2+t)},
\end{equation*}
which yields \eqref{sublinear_growth}.
\end{proof}

\section{Growth of vorticity in $L^p(\mathbb{R}^3)$}
\label{sec5}

We continue with the same initial data $\omega_0$ as in the previous section, which is odd-in-$z$ and nonnegative in the upper quadrant $\Pi^+$. Thoughout this section, the vertical moment \eqref{eq:radial_vertical} plays a crucial role in our analysis, and we impose the additional assumption $Z(0)<\infty$.

Furthermore, the quantity $\omega/r$ is conserved along particle trajectories, and therefore all of its $L^p$ norms are preserved in time. Since $\Pi^{+}$ is invariant by the flow, all the conserved quantities can be restricted to $\Pi^{+}$. In particular, as in the previous section, we denote $A_0:=\|\omega_0/r\|_{L^{\infty}} = \|\omega/r\|_{L^{\infty}}$ and $m_0 := \iint_{\Pi^{+}} \omega_0 \, drdz = \iint_{\Pi^{+}} \omega \, drdz$.

Under this setup, we aim to prove Theorem \ref{thm:main3}. The proof is based on a simple observation: since $Z(t)$ is monotone decreasing in time and positive for all times, the time integral of $-Z'(t)$ in $(0,\infty)$ must be finite (and bounded by $Z(0)$). 
Recall that in \eqref{vertical_1}, we have
\[
-Z' (t) = \int_{0}^{\infty} \frac{u_z(0,z,t)^2}{2} \, dz + \iint_{\Pi^+} \frac{u_r(r,z,t)^2}{r} \, drdz.
\]
The right hand side contains two non-negative integrals. In the following key proposition, we take a closer look at the right hand side and aim to bound it from below. Roughly speaking, it says that if $\omega$ has a positive mass $m$ in the region $\{r<R\}$ at some time, $-Z'(t)$ cannot be too small at this time; it is at least $\sim m R^{-4}$.

\begin{proposition} \label{prop:new}
Take any positive $m\leq m_0$. Then there exists two constants $R_0(m_0, A_0, E_0)\ge 1$ and $c_0(m_0, A_0, E_0, Z(0))>0$, such that the following is true: for any $t\geq 0$ and $R\geq R_0$, if $\omega(\cdot,t)$ satisfies
\begin{equation} \label{eq:new_mass_cond}
\iint_{r\le R} \omega (r,z,t) \, drdz \ge m,    
\end{equation}
then we have
\begin{equation} \label{eq:new}
-Z' (t) =  \int_0^{\infty} \frac{u_z (0,z,t)^2}{2} \, dz + \iint_{\Pi^{+}} \frac{u_r(r,z,t)^2}{r} \, drdz \ge \frac{c_0 m^4}{R^{4}}. 
\end{equation}
\end{proposition}

\begin{proof}[Proof of Proposition \ref{prop:new}]
Assume that the mass condition \eqref{eq:new_mass_cond} holds for some $t$. Fixing this $t$, we split the region $\{r\le R\}$ into $D_1(R)\cup D_2(R)$, where 
\[
D_1(R):= \{(r,z)\in \Pi^{+} \, : \, r< \sqrt{3}z \}\cap \{r\le R\} \quad \text{and} \quad D_2(R):=\{(r,z) \in \Pi^{+} \, :\, r\ge \sqrt{3}z\}\cap \{r\le R\}.
\]
Then we consider two cases. 

\noindent \textit{Case 1.} Suppose that the vorticity mass in $D_1(R)$ exceeds $m/2$, i.e., $\iint_{D_1(R)} \omega \, drdz >m/2$. Then using \eqref{mass1}, we have
\begin{equation} \label{case1_1}
\int_0^{\infty} (-u_z(0,z,t)) \, dz = \iint_{\Pi^+} \frac{z}{\sqrt{r^2+z^2}} \, \omega \, dr dz > \frac{1}{2} \iint_{D_1(R)} \omega \, drdz > \frac{m}{4}.
\end{equation}
On the other hand, note that although $Z(t) = \iint_{\Pi^+} z\omega drdz$ is defined as a double integral in $\Pi^+$, it is identical to a line integral along the $z$-axis: plugging in $\omega=-\partial_z u_r + \partial_r u_z$ and performing integration by parts yields
\[
Z(t) = \iint_{\Pi^+} (-\partial_z u_r + \partial_r u_z) z \, drdz = \int_0^\infty z (-u_z(0,z,t)) \,dz,
\]
where we used that $\iint_{\Pi^+} u_r\, drdz = \iint_{\Pi^+} r^{-1}\partial_z \psi\, drdz = 0$.
From this expression for $Z(t)$ and using its decreasing property, we obtain
\[
\int_{0}^{\infty} z(-u_z(0,z,t)) \, dz = Z(t) \le Z(0).
\]
Therefore, for any $\alpha>0$ we have 
\[
 \int_{\alpha Z(0)}^{\infty} \left(-u_z(0,z,t)\right) \, dz \le \frac{1}{\alpha Z(0)}  \int_{\alpha Z(0)}^{\infty} z \left(-u_z(0,z,t)\right) \, dz \le \frac{1}{\alpha}.
\]
Choosing $\alpha =\frac{8}{m}$ in the above inequality and combining it with \eqref{case1_1} gives
\[
\int_0^{8Z(0)/m}  (-u_z(0,z,t)) \, dz > \frac{m}{4}-\int_{8Z(0)/m}^{\infty}  (-u_z(0,z,t)) \, dz > \frac{m}{8}.
\]
By Cauchy--Schwarz inequality, we now obtain
\[
\int_0^{\infty} u_z(0,z,t)^2 \, dz \ge \int_0^{8Z(0)/m} u_z^2 \, dz \ge \frac{m}{8Z(0)} \left( \int_0^{8Z(0)/m} (-u_z(0,z,t)) \, dz \right)^2 > \frac{m^3}{8^3 Z(0)}.
\]
Since $m^3 \ge m^4/m_0$ (recall that $m\le m_0$), we obtain \eqref{eq:new} under Case 1 by choosing $c_0\leq \frac{1}{2\cdot 8^3 m_0 Z(0)}$ (also note that we have used $R\geq R_0\geq 1$; $R_0$ will be fixed in Case 2).

\medskip 

\noindent \textit{Case 2.} Suppose that the vorticity mass in $D_2(R)$ is at least $m/2$, i.e., $\iint_{D_2(R)} \omega \, drdz \ge m/2$.  In this case, our goal is to show that 
\begin{equation}\label{temp0}
\iint_{\Pi^{+}} \frac{u_r(r,z,t)^2}{r} \, drdz \geq c_0(m_0, A_0, E_0) m^4 R^{-4}.
\end{equation}
In order to do this, we first establish a technical lemma that gives a lower bound of the velocity integral along a portion of the $r$-axis under Case 2.

\begin{lemma} \label{lem:new1}
If at some time $t\geq 0$,  the assumption $\iint_{D_2(R)} \omega(r,z,t) \, drdz \ge \frac{m}{2}$ holds, then we have
\begin{equation} \label{claim_eq}
\int_0^{R} u_r(r,0,t) \, dr \ge c_1 m
\end{equation}
at this time, where $c_1 >0$ is a universal constant.
\end{lemma} 

\begin{proof}[Proof of Lemma \ref{lem:new1}]
Recall that for a point on the $r$-axis, using \eqref{ur}, the velocity $u_r$ is given by \[ u_r(r,0,t) = \frac{2}{\pi} \iint_{\Pi^+} \frac{\overline{z}}{r\sqrt{r\overline{r}}} \, \left(-F'\left(\frac{(r-\overline{r})^2+\overline{z}^2}{r\overline{r}}\right)\right) \overline{\omega} \, d\overline{r}d\overline{z}. \]
Integrating this for $r\in[0,R]$ and swapping the order of integration, we have
\begin{equation} \label{claim_eq1}
\int_0^{R} u_r(r,0,t) \, dr \ge \frac{2}{\pi} \iint_{D_2(R)} \left[ \int_0^{R} \frac{\overline{z}}{r\sqrt{r\overline{r}}} \, \left(-F'\left(\frac{(r-\overline{r})^2+\overline{z}^2}{r\overline{r}}\right)\right) dr \right] \overline{\omega} \, d\overline{r}d\overline{z}. 
\end{equation}

For each $(\overline{r},\overline{z}) \in D_2(R)$, we have $\overline{r} \ge \sqrt{3}\, \overline{z}$. This implies that the interval $(\overline{r}-\overline{z}, \overline{r})$ is contained in $(\overline{r}/3,\overline{r})$, and therefore also in $(0,R)$. Moreover, for every $r\in (\overline{r}-\overline{z}, \overline{r})$, we have \[ \frac{(r-\overline{r})^2+\overline{z}^2}{r\overline{r}} \le \frac{2\,\overline{z}^2}{r\overline{r}} \le \frac{6 \, \overline{z}^2}{\overline{r}^2} \le 2 \quad\text{ for all }(\overline{r},\overline{z})\in D_2(R). \]
In view of the comparison $ -F'(s) \sim \frac{1}{s}$ for $0<s<2$, this yields 
\[
-F'\left(\frac{(r-\overline{r})^2+\overline{z}^2}{r\overline{r}}\right) \ge \frac{c \, \overline{r}^2}{ \overline{z}^2} \quad\text{ for all }(\overline{r},\overline{z})\in D_2(R) \text{ and } r\in (\overline{r}-\overline{z},\overline{r}), 
\]
where $c >0$ is some universal constant.
Thus, we can bound the inner integral on the right hand side of \eqref{claim_eq1} from below as
\[
\int_0^{R} \frac{\overline{z}}{r\sqrt{r\overline{r}}} \, \left(-F'\left(\frac{(r-\overline{r})^2+\overline{z}^2}{r\overline{r}}\right)\right) dr \ge c \int_{\overline{r}-\overline{z}}^{\overline{r}} \frac{\overline{z}}{r\sqrt{r\overline{r}}} \, \frac{\overline{r}^2}{\overline{z}^2} \, dr \ge c \int_{\overline{r}-\overline{z}}^{\overline{r}}  \frac{1}{\overline{z}} \, dr \ge c.
\] 
Combining this with \eqref{claim_eq1}, we conclude that 
\[
\int_0^{R} u_r(r,0,t) \, dr \ge \frac{2c}{\pi} \iint_{D_2(R)} \overline{\omega} \, d\overline{r}d\overline{z} \ge \frac{cm}{\pi},
\] 
and therefore \eqref{claim_eq} follows.
\end{proof}
Note that \eqref{claim_eq} only contains information on the integral of $u_r$ on the $r$-axis, and does not say anything about the interior of $\Pi^+$. In order to obtain \eqref{temp0}, we aim to show a positive lower bound for
$
\int_0^{2R} |u_r(r,\tilde z,t)| \, dr$ for all $\tilde z \in [0,c_2 m^2 R^{-2}]$, where $c_2:= \frac{c_1}{4A_0 m_0}$, $c_1$ is the universal constant in \eqref{claim_eq}, and $A_0:= \|\omega_0/r\|_{L^\infty}$. To do this, note that in the rectangular region 
\[
Q := \{(r,z)\in \Pi^+:  0<r<2 R, \, \, 0<z < c_2 m^2 R^{-2} \}, 
\]
the vorticity mass is bounded above by
\[
\iint_{Q} \omega \, dr dz\le \left\| \frac{\omega}{r} \right\|_{L^{\infty}} \int_{0}^{c_2 m^2 R^{-2}}\int_{0}^{2R}  r \, drdz = 2A_0 c_2 m^2 \leq \frac{c_1 m}{2},
\]
where we used $\|\omega(t)/r\|_{L^\infty}=\|\omega_0/r\|_{L^\infty}=A_0$ in the second step, and the definition $c_2:= \frac{c_1}{4A_0 m_0}$ in the last step. 

As a result, for any $\tilde R \in [R,2R]$ and $\tilde z \in [0, c_2 m^2 R^{-2}]$, the vorticity mass in the rectangle
\[
Q_{\tilde R,\tilde z} := \{(r,z)\in \Pi^+: 0<r<\tilde R, \, \, 0<z < \tilde z \}
\]
does not exceed $c_1 m/2$ since $Q_{\tilde R,\tilde z} \subset Q$. Integrating $\omega$ over $Q_{\tilde R,\tilde z}$ and then applying Green's theorem, we obtain
\begin{equation} \label{case2_1}
\begin{split}
\frac{c_1 m}{2} &\geq \int_{Q_{\tilde R,\tilde z}}\omega\,drdz = \int_{Q_{\tilde R,\tilde z}} (-\partial_z u_r + \partial_r u_z) \,drdz \\
&= \int_0^{\tilde R} u_r(0,r,t)\, dr + \int_0^{\tilde z} u_z(\tilde R, z, t) \, dz - \int_0^{\tilde R} u_r(r,\tilde z,t)\, dr - \int_0^{\tilde z} u_z(\tilde z,0,t)\, dr\\
&= I_1 + I_2 - I_3 - I_4.
\end{split}
\end{equation}
Note that we have $I_1\geq c_1 m$, which follows from \eqref{claim_eq} together with the facts that $\tilde R>R$ and $u_r\geq 0$ on the $r$-axis. Also, we have $I_4<0$ since $u_z<0$ on the $z$-axis due to \eqref{uz}. 

To control $I_2$, we claim the following: there exists some $\tilde R\in [R,2R]$ such that 
\begin{equation} \label{case2_2} 
\int_0^{c_2 m^2 R^{-2}} |u_z(\tilde R, z, t)| \, dz \leq \frac{c_1 m}{4}.
\end{equation}
To show this, assume for contradiction that \eqref{case2_2} is false for all $\tilde R \in [R,2R]$, i.e. 
\[ \int_0^{c_2 m^2 R^{-2}} |u_z(r,z,t)| \, dz > \frac{c_1 m}{4} \quad \text{ for all }r\in [R,2R].\] 
Then, noting $E_0\ge 2\iint_{\Pi^+} ru_z^2 \, drdz$ and applying Cauchy--Schwarz inequality, we get 
\begin{align*}
E_0 &\ge 2\int_{R}^{2R} \int_0^{c_2 m^2 R^{-2}} r|u_z(r,z,t)|^2 \, dz dr \\
&\ge \frac{2R^2}{c_2 m^2} \int_{R}^{2R} r \left(\int_0^{c_2 m^2 R^{-2}} |u_z(r,z,t)| \, dz \right)^2 dr > \frac{2R^2}{c_2 m^2} \cdot 3R^2\Big(\frac{c_1 m}{4}\Big)^2 \geq \frac{3 c_1^2} {8c_2} R_0^4,
\end{align*}
where in the last step we used the assumption $R\geq R_0$. We can then choose $R_0 := R_0(m_0, A_0, E_0)\geq 1$ such that the right hand side above exceeds $2E_0$ (recall that $c_2$ only depends on $m_0$ and $A_0$, and $c_1$ is the universal constant in \eqref{claim_eq}), arriving at a contradiction. This finishes the proof of \eqref{case2_2}.

Note that \eqref{case2_2}  implies that for all $\tilde z \in [0,c_2 m^2 R^{-2}]$, there exists some $\tilde R \in [R,2R]$ such that $|I_2| \leq \frac{c_1 m}{4}$, therefore for such $\tilde R$ we have
\[
\int_0^{\tilde R} u_r(r,\tilde z,t)\, dr = I_3 \geq I_1 - |I_2| - I_4 - \frac{c_1 m}{2} \geq \frac{c_1 m}{4}.
\]
Since $\tilde R \leq 2R$, this implies
\[
\int_0^{2R} |u_r(r,\tilde z,t)|\, dr \geq I_3 \geq \frac{c_1 m}{4} \quad \text{ for all }\tilde z \in [0,c_2 m^2 R^{-2}].
\]
Consequently, applying Cauchy--Schwarz inequality yields
\begin{align*}
\iint_{\Pi^+} \frac{u_r(r,z,t)^2}{r} \, drdz &\ge \int_{0}^{c_2 m^2 R^{-2}}  \left(\int_{0}^{2R} \frac{u_r(r,z,t)^2}{r} \, dr \right) dz \\
&\ge \frac{1}{2R^2} \int_{0}^{c_2 m^2 R^{-2}}  \left(\int_{0}^{2R} |u_r(r,z,t)| \, dr \right)^2 dz \\
&\geq \frac{1}{2R^2} \, c_2 m^2 R^{-2} \left(\frac{c_1 m}{4}\right)^2 = \frac{c_2 c_1^2}{32}m^4 R^{-4}.    
\end{align*}
We therefore obtain \eqref{temp0} and consequently \eqref{eq:new}, completing the proof for Case 2. 
\end{proof}

We now prove Theorem \ref{thm:main3}.

\begin{proof}[Proof of Theorem \ref{thm:main3}]
Recall that the quantity $\omega/r$ is transported by the flow $u$, which is divergence-free in $\mathbb{R}^3$. Therefore, the distribution function of $\omega/r$ in $\Pi^+$ remains the same for all time with respect to the measure $d\mu := r\, drdz$.
As before, we also recall notations $A_0:=\|\omega_0/r\|_{L^{\infty}} >0$ and $m_0:=\iint_{\Pi^+} \omega_0 \, drdz>0$. 

For each $0<\lambda\le A_0$, defining the level sets $D_{\lambda}(t):=\{(r,z) \in \Pi^+ \, : \, \omega (r,z,t)/r \ge \lambda\}$, we observe that
\[
m_0 = \int_{\Pi^{+}} \omega (r,z, t) \, drdz = \int_{\Pi^{+}} \frac{\omega(r,z, t)}{r} \, rdrdz = \int_0^{A_0} \mu (D_{\lambda}(t)) \, d\lambda>0,
\]
where $\mu (A):=\int_{A} \, rdrdz$. Thus, there exists $0<\delta \le A_0$ that only depends on the distribution of $\omega_0/r$, such that 
\begin{equation}\label{ineq_D}
 \mu (D_{\delta} (t)) = \mu (D_{\delta} (0)) \ge \frac{m_0}{2A_0} \quad\text{ for all } t\geq 0.
 \end{equation}

From now on, we fix 
\begin{equation}\label{def_m}
m:=\min\Big\{m_0, \frac{m_0 \delta}{4A_0}\Big\} \in (0,m_0],
\end{equation} and let $R_0\ge 1$ and $c_0>0$ be the constants in Proposition~\ref{prop:new} corresponding to this $m$ (thus, the constants $m, R_0, c_0$ only depend on the initial data $\omega_0$ and $u_0$). Applying Proposition~\ref{prop:new} with this $m$, we see that for any taken $R\ge R_0$, if \eqref{eq:new_mass_cond} is satisfied at some $t\geq 0$, then we have \eqref{eq:new} at this time $t$. Note that this implies \eqref{eq:new_mass_cond} cannot hold for too long time. Indeed, if we define $C_7:=c_0^{-1}  m^{-4} Z(0)$, which only depends on the initial data, and suppose that \eqref{eq:new_mass_cond} holds for all $t\in [0, C_7 R^4]$, then we would have
\[
-Z'(t) \geq c_0 m^4 R^{-4} \quad\text{ for all } t\in [0, C_7 R^4].
\]
Integrating this inequality in time yields
\[
Z(0) - Z(C_7 R^4) \geq \int_{0}^{C_7R^4} c_0 m^4 R^{-4} \, dt = c_0 m^4 C_7 = Z(0),
\]
i.e., $Z(C_7 R^4) \leq 0$, contradicting with the fact that $Z(t) >0$ for all $t\geq 0$. As a result, we have shown that for any $R\ge R_0$, there must be some $t_R\in [0, C_7 R^4]$ such that \eqref{eq:new_mass_cond} fails. Hence at this $t_R$ we have
\begin{equation}\label{ineq_m}
\iint_{r\leq R}\omega(r,z,t_R) \, drdz < m.
\end{equation}

On the other hand, the left hand side above can be bounded from below as
\[
\iint_{r\leq R}\omega(r,z,t_R) \, drdz= \iint_{r\leq R}\frac{\omega(r,z,t_R)}{r} \, d\mu \geq \iint_{D_\delta(t_R) \cap \{r\leq R\}}\frac{\omega(r,z,t_R)}{r} \, d\mu \geq \delta \mu(D_\delta(t_R) \cap \{r\leq R\}).
\]
Combining this with \eqref{ineq_m} and the definition of $m$ in \eqref{def_m} implies
\[
\mu(D_\delta(t_R) \cap \{r\leq R\}) < \frac{m}{\delta} \leq \frac{m_0}{4A_0}.
\]
Subtracting the above inequality from \eqref{ineq_D}, we obtain
\begin{equation}\label{ineq_tr}
\mu(D_\delta(t_R) \cap \{r> R\}) > \frac{m_0}{4A_0} \quad \text{ for some }t_R \in [0, C_7 R^4].
\end{equation}
It remains to show that this implies the growth of $\|\omega\|_{L^p}$ (in the limsup sense) for all $1\leq p\leq \infty$.

For $p=\infty$, note that \eqref{ineq_tr} directly implies 
\[
\|\omega(t_R)\|_{L^\infty(\mathbb{R}^3)} = \|\omega(t_R)\|_{L^\infty(\Pi^+)} \geq \delta R,
\]
by the definition of $D_\delta$ and the fact that $D_\delta(t_R) \cap \{r>R\} \neq \emptyset$. 

For $1\leq p<\infty$, we have
\begin{align*}
		\|\omega(t_R)\|_{L^{p}(\mathbb{R}^3)}^p &= 4\pi \iint_{\Pi^+} |\omega (r,z,t_R)|^p \, d\mu \ge 4\pi R^p \iint_{D_{\delta} (t_R) \cap \{ r > R\}} \left|\frac{\omega(r,z,t_R)}{r}\right|^p \, d\mu \\
		&\ge 4\pi \left(\delta R\right)^p \mu(D_{\delta} (t_R) \cap \{ r > R\}) > \frac{\pi m_0}{A_0} \left(\delta R\right)^p,
		\end{align*}
and taking the power $1/p$,
\[
\|\omega(t_R)\|_{L^{p}(\mathbb{R}^3)} \geq (\pi m_0 A_0^{-1})^{1/p} \delta R.
\]
Since $\delta>0$ depends only on the initial data and $(\pi m_0 A_0^{-1})^{1/p}$ is uniformly bounded from below for all $p\geq 1$, we conclude that, at time $t_R$,
\[
\|\omega(t_R)\|_{L^p(\mathbb{R}^3)} \geq \tilde c (\omega_0,u_0) R \quad\text{ for all } 1\leq p\leq \infty,
\]
with some constant $\tilde c (\omega_0,u_0)>0$ that depends only on the initial data. 

As a consequence, for all $R\geq R_0$ we have
\[
\sup_{t\in [0,C_7 R^4]} \frac{\|\omega(t)\|_{L^p(\mathbb{R}^3)}}{R} \geq \tilde c (\omega_0,u_0).
\]
Through a change of variable $T:=C_7 R^4$, we obtain that for all $T\geq C_7 R_0^4$,
\[
\sup_{t\in [0,T]} \frac{\|\omega(t)\|_{L^p(\mathbb{R}^3)}}{t^{1/4}} \geq \sup_{t\in [0,T]} \frac{\|\omega(t)\|_{L^p(\mathbb{R}^3)}}{T^{1/4}} \geq \tilde c (\omega_0,u_0) C_7^{-1/4} =: c (\omega_0,u_0)>0,
\]
which yields \eqref{limsup_growth}.
\end{proof} 

Finally, we prove Theorem~\ref{thm:main4}, whose proof is almost a direct consequence of Proposition~\ref{prop:new}.
\begin{proof}[Proof of Theorem \ref{thm:main4}]
Fix any $R\geq R_0$. For each $t\ge 0$, recall that $m_R(t) := \iint_{\Pi_{+} \cap \{r\le R\}} \omega(r,z,t) \,drdz$ satisfies $m_R(t)\leq m_0$ by definition.

Therefore, for each $t\geq 0$, by defining $m$ as $m_R(t)$ and applying Proposition~\ref{prop:new} with this $m$, we obtain
\[
-Z'(t) \geq c_0(\omega_0,u_0) m_R(t)^4 R^{-4} \quad\text{ for all }t\geq 0.
\]
Integrating in time over $t\in[0,\infty)$ directly leads to
\[
\int_0^\infty m_R(t)^4 dt \leq c_0(\omega_0,u_0)^{-1} Z(0) R^4 < \infty, 
\]
which finishes the proof.
\end{proof}

%%%%%%%%%%%%%%%%%%%%%%%%%%%%%%%%%%%%%%%%%%%	

\end{document}